\documentclass[a4paper,12pt]{article}
\usepackage[text={6.5in,9in},centering]{geometry}
\usepackage{graphicx,url,subfig}
\RequirePackage[OT1]{fontenc}
\RequirePackage{amsthm,amsmath,amssymb,amscd}
\RequirePackage[numbers,sort&compress]{natbib}
\RequirePackage[colorlinks,citecolor=blue,urlcolor=blue]{hyperref}
\usepackage{enumerate}
\usepackage{xfrac}
\usepackage{datetime}
\usepackage{etex}

\usepackage{flexisym}

\usepackage[T1]{fontenc}
\usepackage[utf8]{inputenc}
\usepackage{authblk}

\makeatother
\numberwithin{equation}{section}
\allowdisplaybreaks

\theoremstyle{plain}
\newtheorem{theorem}{Theorem}[section]

\newtheorem{lemma}[theorem]{Lemma}
\newtheorem{proposition}[theorem]{Proposition}

\theoremstyle{definition}
\newtheorem{definition}[theorem]{Definition}

\newtheorem{remark}[theorem]{Remark}

\newcommand{\E}{\mathbb{E}}

\newcommand{\W}{\dot{W}}
\newcommand{\ud}{\ensuremath{\mathrm{d}}}

\newcommand{\Norm}[1]{\left|\left|  #1   \right|\right|}

\newcommand{\InPrd}[1]{\left\langle #1 \right\rangle}

\newcommand{\calH}{\mathcal{H}}

\newcommand{\bbC}{\mathbb{C}}
\newcommand{\bbN}{\mathbb{N}}
\newcommand{\bbP}{\mathbb{P}}

\newcommand{\R}{\mathbb{R}}

\title{Intermittency for the stochastic heat equation driven by a rough time fractional Gaussian noise}
\author[$^*$]{Le Chen}
\author[$^*$]{Yaozhong Hu}
\author[$^\dag$]{Kamran Kalbasi}
\author[$^*$]{David Nualart}
\affil[$^*$]{Department of Mathematics, University of Kansas\thanks{$^*$Hu is partially supported by a grant from the Simons Foundation \#209206;
Nualart is partially supported by the NSF grant DMS1512891 and the ARO grant FED0070445.}}
\affil[$^\dag$]{Mathematics Institute, University of Warwick\thanks{$^\dag$Kalbasi is supported by a fellowship from the Swiss National Science Foundation.}}

\renewcommand\footnotemark{}

\date{}
\begin{document}
\maketitle
\begin{center}
\begin{minipage}[rct]{5 in}
\footnotesize \textbf{Abstract:}
This paper studies the stochastic heat equation driven by time fractional Gaussian noise
with Hurst parameter $H\in(0,1/2)$.
We establish the Feynman-Kac representation of the solution and
use this representation to obtain matching lower and upper bounds for the $L^p(\Omega)$ moments of the solution.

\vspace{2ex}

\vspace{2ex}
\vspace{4ex}
\end{minipage}
\end{center}

% \tableofcontents
\setlength{\parindent}{1.5em}

%%%%%%%%%%%%%%%%%%%%%%%%%%%%%%%%%%%%%%%%%%%%%%
%%%%% MAIN:
%%%%%%%%%%%%%%%%%%%%%%%%%%%%%%%%%%%%%%%%%%%%%%

\section{Introduction}
As pointed out by Zel{\textprime}dovich {\it et al} \cite[p. 237]{Zel90},
{\it intermittency} is a universal phenomenon provided that a random field is of multiplicative type.
Intermittency is characterized by enormous growth rates of moments of the random field
and it has been intensively studied in the past two decades
for stochastic partial differential equations of various kinds; see, e.g., \cite{BC15AOP,BertiniCancrini95,CarmonaMolchanov94PAM,ChenDalang13Heat,ChenHuNualart15,CJKS13,FK08Int,HHNT15EJP}.
These growth rates both depend on the noise structures \cite{BC15AOP,HHNT15EJP} and also on the partial differential operators \cite{ChenHuNualart15,FK08Int}.
In the literature, the noise is either white in time \cite{BertiniCancrini95,CarmonaMolchanov94PAM,ChenDalang13Heat,ChenHuNualart15,CJKS13,FK08Int}
or more regular than the white noise \cite{BC15AOP,HHNT15EJP}.
Little is known about the intermittency for the case when the noise in time is rougher than the white noise.
This latter fact motivates this current investigation.
In particular, we will study in this paper the intermittency property for the following stochastic heat equation
subject to a noise which is rougher than the white noise in time,
\begin{align}\label{E:SHE}
 \begin{cases}
  \displaystyle \frac{\partial }{\partial t} u(t,x)  = \frac{1}{2} \Delta u(t,x) + u(t,x)\frac{\partial }{\partial t}W(t,x),& t>0,x\in\R^d,\\[0.5em]
  u(0,x) = u_0(x),
 \end{cases}
\end{align}
where $u_0$ is a bounded measurable function. $W=\{W(t,x),\: t\ge 0, x\in\R^d\}$
is a Gaussian random field, which is fractional Brownian motion of Hurst parameter $H\in (0,1/2)$ in time
and has correlation in space given by $Q(x,y)$:
\[
\E\left[W(t,x)W(s,y)\right]= \frac{1}{2}\left(t^{2H}+s^{2H}-|t-s|^{2H}\right) Q(x,y).
\]
We assume that $Q(x,y)$ satisfies the following two conditions:
\begin{enumerate}[(\bf H1)]
 \item There exist some constants $\alpha\in (0,1]$ and $C_0>0$ such that
\begin{align}\label{E:H1}\tag{H1}
% \left|Q(x,u) + Q(y,w)-Q(x,w)-Q(y,u)\right|\le C_{0} |x-y|^{\alpha}|u-w|^{\alpha},
Q(x,x)+Q(y,y)-2Q(x,y)\le C_0 |x-y|^{2\alpha},\quad\text{for all $x$ and $y\in\R^d$.}
\end{align}
\item There exist some constants $\beta\in [0,1)$ and $C_2>0$ such that for all $M>0$,
\begin{align}\label{E:H2}\tag{H2}
Q(x,y)\ge C_2 M^{2 \beta}, \qquad \text{for all $x,\: y\in\R^d$ with $x_i,\: y_i\ge M$, $i=1,\dots,d$.}
\end{align}
\end{enumerate}

It is known that {\it Feynman-Kac formula/representation} for the solution is
a powerful tool for studying the moments of the solution; see \cite{CarmonaMolchanov94PAM,CJKS13,HHNT15EJP}.
Hence, the first challenging problem in this paper is to establish the following Feynman-Kac formula for the solution to \eqref{E:SHE}:
% This representation usually takes the following form:
\begin{align}\label{E:FKF}
u(t,x)= \E^{B}\left[u_0(B_t^x)\exp\int_0^t W(\ud s,B_{t-s}^x)\right],
\end{align}
where $B=\left\{B_t^x=B_t+x,\: t\ge 0,x\in\R^d\right\}$ is a $d$-dimensional Brownian motion starting from $x\in\R^d$,
independent of $W$, and the expectation is with respect to the Brownian motion.
Hu {\it et al.} \cite{HLN12AOP} established this representation \eqref{E:FKF} for the case where $H\in (1/4,1/2)$.
In this paper we will improve their results by allowing the Hurst parameter $H$ to be any value in $(0,1/2)$.

More precisely, we will show that, for any $H\in(0,1/2)$, if condition \eqref{E:H1} holds and $2H+\alpha>1$,
then the solution to \eqref{E:SHE} is given by \eqref{E:FKF}.
Moreover, using this representation \eqref{E:FKF}, we are able to show that
for some nonnegative constants $\overline{C}$ and $\underline{C}$, 
the solution to \eqref{E:SHE} satisfies the following moment bounds
\begin{align}\label{E:uplowBd}
\underline{C}\exp\left(\underline{C} k^{\frac{2-\beta}{1-\beta}} t^{\frac{2H+\beta}{1-\beta}}\right)\le
\E\left[u(t,x)^k\right]\le
\overline{C}\exp\Bigg(\:\overline{C} k^{\frac{2-\alpha}{1-\alpha}} t^{\frac{2H+\alpha}{1-\alpha}}\:\:\Bigg)
\end{align}
for large $t$ and $k$, where we need to assume condition \eqref{E:H2} and $\inf_{x\in\R^d}u_0(x)>0$ to establish the lower bound.
When $\alpha=\beta$ (see Remark \ref{Example} below for one example),
our exponents in \eqref{E:uplowBd} are sharp in the sense that one can define the {\it moment Lyapunov exponents}
\[
\overline{m}_k(x):=
\limsup_{t\rightarrow +\infty}t^{-\frac{2H+\alpha}{1-\alpha}}\log\E\left[u(t,x)^k\right]\quad\text{and}\quad
\underline{m}_k(x):=
\liminf_{t\rightarrow +\infty}t^{-\frac{2H+\alpha}{1-\alpha}}\log\E\left[u(t,x)^k\right],
\]
and establish easily from \eqref{E:uplowBd} that
\begin{align}\label{E:Lyapunov}
\underline{C} k^{\frac{2-\alpha}{1-\alpha}} \le
\inf_{x\in\R^d}\underline{m}_k(x)
\le
\sup_{x\in\R^d}\overline{m}_k(x)\le
\overline{C} k^{\frac{2-\alpha}{1-\alpha}},\quad\text{for all $k\ge 2$.}
\end{align}
% provided that $u_0$ is bounded and .
Therefore, this solution is {\it fully intermittent} \cite[Definition III.1.1]{CarmonaMolchanov94PAM}.

\begin{remark}\label{Example}
If $d=1$ and $Q(x,y)$ is the covariance of a fractional Brownian
motion $\{\:B_x^{\Theta},\:x\in\R\}$ with Hurst parameter $\Theta\in (0,1)$, i.e.,
\[
Q(x,y)= \E\left[B_x^{\Theta} B_y^{\Theta}\right]=\frac{1}{2}\left(|x|^{2\Theta}+|y|^{2\Theta}-|x-y|^{2\Theta}\right),
\]
then it is easy to see that both conditions \eqref{E:H1} and \eqref{E:H2} are satisfied with
$\alpha = \beta =\Theta$ and \eqref{E:uplowBd} becomes
\begin{align}\label{E2:uplowBd}
 \underline{C}\exp\left(\underline{C} \: k^{\frac{2-\Theta}{1-\Theta}}
 t^{\frac{2H+\Theta}{1-\Theta}}\right)\le
\E\left[u(t,x)^k\right]\le
\overline{C}\exp\left(\overline{C} \: k^{\frac{2-\Theta}{1-\Theta}}
t^{\frac{2H+\Theta}{1-\Theta}}\right).
\end{align}
\end{remark}

% Thanks to \eqref{E:uplowBd}, if $\beta=\max\left(\alpha,\alpha\right) =:\lambda$,

\bigskip
There is an extensive literature on the Feynman-Kac formula for stochastic
partial differential equations under various random potentials.
We refer interested readers to the references in \cite{HLN12AOP,HN09PTRF,HNS09}.
Hu {\it et al.} \cite{HNS09} proved that if the random potential $W=\{W(t,x),t\ge 0, x\in\R^d\}$ is a fractional
Brownian sheet with Hurst parameter $(H_0,H_1,\dots,H_d)$ that satisfies
\begin{align}\label{E:Cond>1/2}
H_i\in(1/2,1),\: i=1,\dots,d, \quad \text{and}\quad
2H_0+\sum_{i=1}^d H_i>d+1,
\end{align}
then the solution to the following stochastic heat equation
\begin{align}
\label{E:SHE2}
 \begin{cases}
  \displaystyle \frac{\partial }{\partial t} u(t,x)  = \frac{1}{2} \Delta u(t,x) +
  u(t,x)\frac{\partial^{d+1} }{\partial t\partial x_1\cdots \partial x_d}W(t,x),& t>0,\: x\in\R^d,
  \\[0.5em]
  u(0,x) = u_0(x),
 \end{cases}
\end{align}
admits a Feynman-Kac representation
\[
u(t,x)=\E^{B}\left[u_0(B_t^x)\exp\left(\int_0^t \int_{\R^d}\delta\left(B_{t-s}^x-y\right)W(\ud s,\ud y)\right)\right],
\]
where $B$ is a $d$-dimensional Brownian motion (the same as $B$ in \eqref{E:FKF}), independent of $W$.
In this framework, condition \eqref{E:Cond>1/2} implies that $H_0>1/2$.

In order to handle the case where $H_0<1/2$, one may impose better spatial correlations.
When $H_0\in (1/4,1/2)$, Hu {\it et al.} \cite{HLN12AOP} established the Feynman-Kac representation for
\eqref{E:FKF} with a similar spatial covariance $Q(x,y)$ that satisfies a growth condition (see \eqref{E:H3} below) and is locally $\gamma$-H\"older continuous with
$\gamma>2-4H_0$.
Notice that the fact that $Q$ is a covariance function implies that there exists a Gaussian process $Y=\{Y(x),x\in\R^d\}$ such that
$Q(x,y)=\E[Y(x)Y(y)]$. Then it is natural to assume some sample path regularity of $Y$ through the following condition
\begin{align}\label{E:QY}\tag{H1'}
\E\left[\left(Y(x)-Y(y)\right)^2\right]\le C_0|x-y|^{2\alpha}.
\end{align}
Because $Y$ is Gaussian, \eqref{E:QY} implies that $Y$ is a.s. $\gamma$-H\"older continuous for all $\gamma<\alpha$.
Clearly the two conditions \eqref{E:QY} and \eqref{E:H1} are equivalent.
Then under \eqref{E:H1} (or equivalently \eqref{E:QY}),
we are able to establish the Feynman-Kac formula for any $H_0\in ((1-\alpha)/2,1/2)$.
Note that $\alpha$ can be arbitrarily close to one by choosing $Q$ properly; see Remark \ref{Example} for an example.

The above representation of $Q$ using $Y$ implies a growth condition of $Q$,
which is listed below for the convenience of later reference,
\begin{enumerate}
\item[(\bf H3)] There exists a constant $C_1>0$ such that for all $M>0$,
\begin{align}\label{E:H3}\tag{H3}
|Q(x,y)| \le C_1 (1+M)^{2\alpha}, \quad\text{for all $x,y\in\R^d$ with $|x|, |y|\le M$.}
\end{align}
\end{enumerate}

When the space $\R^d$ is replaced by $\mathbb{Z}^d$ in \eqref{E:SHE},
the Brownian motion $B$ in \eqref{E:FKF} should be replaced by a locally constant random walk.
% which is very different  the Brownian motion.
Kalbasi and Mountford \cite{KM15} recently studied this case and established
the Feynman-Kac formula for any $H_0\in(0,1)$.

\bigskip
It is interesting, even formally, to compare the exponents obtained in this work with the previous ones.
% As for the exponents in the moment bounds,
Hu {\it et al.} \cite{HHNT15EJP} recently studied \eqref{E:SHE2}
with the noise having the following covariance form
\begin{align}\label{E:ColorNoise}
\E\left[\W(t,x)\W(s,y)\right]=\gamma(t-s)\Lambda(x-y),
\end{align}
where $\W:=\frac{\partial^{d+1}\:\:W }{\partial t\partial x_1\cdots \partial x_d}$;
see also a closely related work by Balan and Conus \cite{BC15AOP}.
Under the condition that for some constants $c_0$, $C_0$, $c_1$, $C_1$, $\kappa\in(0,1)$ and $\sigma\in(0,2)$, 
\begin{align}\label{E:CondHu}
c_0 |t|^{-\kappa} \le \gamma(t)\le C_0|t|^{-\kappa}\quad\text{and}\quad
c_1 |x|^{-\sigma} \le \Lambda(x)\le C_1|x|^{-\sigma},
\end{align}
it is proved in \cite{HHNT15EJP} that 
% then the solution to \eqref{E:SHE2} satisfies the following moment bounds
\begin{align}\label{E:Lyapunov2}
\underline{C}\exp\left(\underline{C} \: k^{\frac{4-\sigma}{2-\sigma}} t^{\frac{4-2\kappa-\sigma}{2-\sigma}} \right)
\le
\E\left[u(t,x)^k\right]
\le
\overline{C}\exp\left(\overline{C} \: k^{\frac{4-\sigma}{2-\sigma}} t^{\frac{4-2\kappa-\sigma}{2-\sigma}} \right).
\end{align}
The noises for both equations \eqref{E:SHE} and \eqref{E:SHE2} (with noise \eqref{E:ColorNoise})
are similar in time.
Our noise formally corresponds to the case $\kappa=2-2H$.
% , even through the condition $\kappa\in (0,1)$ imposed in \cite{HHNT15EJP} implies that $H\in (1/2,1)$.
Comparing \eqref{E:H1} with \eqref{E:CondHu}, we should have $\sigma=-2\alpha$.
However, after substituting $\kappa$ by $2-2H$ in the exponents of \eqref{E:Lyapunov2} and comparing the following two exponents, 
\[
k^{\frac{2-\alpha}{1-\alpha}}\:\: t^{\frac{2H + \alpha}{1-\alpha}}\:\:\text{in \eqref{E:uplowBd}}
\qquad
\text{and}
\qquad
k^{\frac{2-\sigma/2}{1-\sigma/2}}\: t^{\frac{2H - \sigma/2}{1-\sigma/2}}
\:\text{in \eqref{E:Lyapunov2},}
\]
we immediately see some mismatches of the sign if $\sigma=-2\alpha$. This is due to the different natures of these two noises in space.
Our noise in space is nonhomogeneous and the function $x\mapsto Q(x,x)$ is finite at the origin but has a growth rate at infinity.
On the other hand, the noise with $\Lambda$ in \eqref{E:ColorNoise} is homogeneous and it is singular at the origin
but decreases to zero at infinity.
Nevertheless, in both cases, the exponents of $k$ depend only on the spatial correlations.
Moreover, when $\sigma=1$ (noise is white in space for \eqref{E:SHE2})
and $\Theta=1/2$ in \eqref{E2:uplowBd} (the case when $Q$ is a correlation function of a Brownian motion; see Remark \ref{Example}),
both exponents of $k$ are equal to $3$.
If one would like to write the two exponents in one formula, we may use
\[
k^{\frac{2-|\alpha|}{1-|\alpha|}} t^{\frac{2H+\alpha }{1-|\alpha|}}
\qquad\text{or}\qquad
k^{\frac{2-|\sigma/2|}{1-|\sigma/2|}} t^{\frac{2H- \sigma/2 }{1-|\sigma/2|}}\:.
\]

\bigskip
Throughout this paper, denote $\alpha_H = 2H(2H-1)$, which is negative for $H\in(0,1/2)$.
For $t, \: s \in\R$, denote
\begin{align}\label{E:RH}
R_H(t,s):=\frac{1}{2}\left(|t|^{2H}+|s|^{2H}-|t-s|^{2H}\right).
\end{align}
Let $\Norm{\cdot}_\kappa$ be the $\kappa$-H\"older norm and $C^\kappa([0,T])$ be the set of
$\kappa$-H\"older continuous functions on $[0,T]$.

\bigskip
This paper is organized as follows:
In Section \ref{S:SI}, we define the stochastic integral in \eqref{E:FKF} through approximation
and derive some properties of this stochastic integral.
In Section \ref{S:Upper}, we first make sense of expression \eqref{E:FKF} by showing that the stochastic integral in \eqref{E:FKF}
has exponential moments. As a consequence, we derive the upper bound of \eqref{E:uplowBd}.
Then we validate that \eqref{E:FKF} is a weak solution to \eqref{E:SHE}.
The lower bound in \eqref{E:uplowBd} is proved in Section \ref{S:Low}.
Finally, some technical lemmas are proved or listed in Appendix.

% \section{Some preliminaries}

\section{Stochastic integral with respect to \texorpdfstring{$W$}{}}\label{S:SI}
In this section, we introduce the stochastic integral with respect to
$W$ that appears in \eqref{E:FKF} and prove some useful properties.
The integral is defined through an approximation scheme,
which requires an extension of the noise $W$ from $t\ge0$ to $t\in\R$, i.e.,
$W=\{W(t,x),\: t\in\R, \: x\in\R^d\}$ is a mean zero Gaussian process
with the following covariance
\[
\E\left[W(t,x)W(s,y)\right] = R_H(t,s)Q(x,y),\quad
\text{for all $t,s\in\R$ and $x,y\in\R^d$.}
\]

\begin{definition}\label{D:SI}
Given a continuous function $\phi$ on $[0,T]$, define
\[
\int_0^t W(\ud s,\phi_s) := \lim_{\epsilon\rightarrow 0}\int_0^t \W^\epsilon(s,\phi_s)\ud s,
\]
if the limit exists in $L^2(\Omega)$, where
\begin{align} \label{E:WEpsilon}
 \W^\epsilon(s,x)=(2\epsilon)^{-1}\left(W(s+\epsilon,x)-W(s-\epsilon,x)\right).
\end{align}
\end{definition}
The aim of this section is the following theorem and proposition.
Denote
\[
\widehat{Q}(u,v, \phi,\psi) =
\frac{1}{2}\left[Q(\phi_u,\psi_u)+Q(\phi_v,\psi_v) - Q(\phi_u,\psi_v) - Q(\phi_{v},\psi_{u}) \right].
\]

\begin{theorem}\label{T:Def}
Assume that $Q$ satisfies condition \eqref{E:H1}.
Then for all $0<t\le T$ and $\phi,\: \psi\in C^\kappa([0,T])$ with $\alpha\kappa+H>1/2$, the stochastic integral $I(\phi):=\int_0^t W(\ud s, \phi_s)$ exists and
\begin{align}\label{E:SecM}
\begin{aligned}
% \E\left[\int_0^t W(\ud s,\phi_s)\int_0^t W(\ud s,\psi_s)\right]
\E\left[I(\phi)I(\psi)\right]
=&
\quad H \int_0^t \theta^{2H-1} \left[Q(\phi_\theta,\psi_{\theta})+Q(\phi_{t-\theta},\psi_{t-\theta})\right] \ud \theta
\\
&
-\alpha_{H}\int_0^t\int_0^\theta
r^{2H-2} \widehat{Q}(\theta, \theta-r, \phi,\psi) \ud r \ud \theta.
\end{aligned}
\end{align}
Moreover,
\begin{align}\notag
% \left|\E\left[\int_0^t W(\ud s,\phi_s)\int_0^t W(\ud s,\psi_s)\right]\right|\le &
\left|\E\left[I(\phi)I(\psi)\right]\right|\le &
\quad H \int_0^t \theta^{2H-1}\left[Q(\phi_\theta,\psi_\theta)+Q(\phi_{t-\theta},\psi_{t-\theta})\right]
\ud \theta \\
& + \frac{|\alpha_H| C_0}{2}\int_0^t\int_0^\theta r^{2H-2} |\phi_{\theta}-\phi_{\theta-r}|^\alpha
|\psi_{\theta}-\psi_{\theta-r}|^\alpha \ud r\ud \theta
\label{E:BddM-Int}
\\
=&
C_{\phi,\psi}t^{2(H+\alpha\kappa)}
+ C_{\phi,\psi}^* t^{2H},
\label{E:BddM}
\end{align}
where
\[
C_{\phi,\psi} := \frac{H(1-2H)C_0\Norm{\phi}_\kappa^{\alpha}\Norm{\psi}_\kappa^{\alpha}}{2(H+\alpha\kappa)(2(H+\alpha\kappa)-1)}\quad\text{and}\quad
C_{\phi,\psi}^* := C_1(1+\Norm{\phi}_\infty\vee \Norm{\psi}_\infty)^{2\alpha},
\]
and the constants $C_0$ and $C_1$ are defined in \eqref{E:H1} and \eqref{E:H3}, respectively.
\end{theorem}
%  beta
\begin{remark}
By symmetry,
\begin{align}
\begin{aligned}
\int_0^t\int_0^\theta&
r^{2H-2} \widehat{Q}(\theta, \theta-r, \phi, \psi) \ud r \ud \theta
=
\frac{1}{2}\int_0^t\int_0^t
|u-v|^{2H-2} \widehat{Q}(u,v,\phi,\psi) \ud u \ud v.
\end{aligned}
\end{align}
Therefore, \eqref{E:SecM} can be equivalently written as
\begin{align}\label{E:Formal}
\begin{aligned}
 \E[I(\phi)I(\psi)]=&
 H \int_0^t Q(\phi_s,\psi_s)\left[s^{2H-1}+(t-s)^{2H-1}\right]\ud s\\
 &+ \frac{|\alpha_H|}{2}\int_0^t\int_0^t
 |u-v|^{2H-2} \widehat{Q}(u,v,\phi,\psi)\ud u\ud v,
\end{aligned}
\end{align}
and similarly,
\begin{align}\label{E:SecM-Sym}
\begin{aligned}
\int_0^t\int_0^\theta r^{2H-2} &|\phi_{\theta}-\phi_{\theta-r}|^\alpha
|\psi_{\theta}-\psi_{\theta-r}|^\alpha \ud r\ud \theta\\
&=
\frac{1}{2}\int_0^t\int_0^t |u-v|^{2H-2} |\phi_{u}-\phi_{v}|^\alpha
|\psi_{u}-\psi_{v}|^\alpha \ud u\ud v.
\end{aligned}
\end{align}
\end{remark}

\begin{proposition}\label{P:Holder}
Suppose $\phi\in C^\kappa([0,T])$ with $\alpha\kappa+H>1/2$. Then for all $0\le s<t\le T$,
\begin{align}\label{E:Holder}
\begin{aligned}
 \E\left[\left(\int_0^t W(\ud r,\phi_r)-\int_0^s W(\ud r,\phi_r)\right)^2\right]
 \le & \quad C' \left(1+\Norm{\phi}_\infty\right)^{2\alpha} (t-s)^{2H}\\
 & + C'' \Norm{\phi}_\kappa^{2\alpha} (t-s)^{2(H+\alpha\kappa)},
\end{aligned}
\end{align}
where the constants $C'$ and $C''$ depend on $H$, $T$, $\alpha$ and $\kappa$.
As a consequence, the process $X_t=\int_0^t W(\ud r,\phi_r)$ is almost
surely $(H-\epsilon)$-H\"older continuous for any $\epsilon>0$.
\end{proposition}

\bigskip
The proofs of Theorem \ref{T:Def} and Proposition \ref{P:Holder} require some lemmas.
% We first introduce some notation: beta
Denote
\[
I_\epsilon(\phi)=\int_0^t\W^\epsilon(s, \phi_s)\ud s.
\]
By (3.2) of \cite{HLN12AOP}, for $\phi,\psi\in C([0,T])$,
\begin{align}
\begin{aligned}
\E\left[I_\epsilon(\phi)I_\delta(\psi)\right] &
=
\int_0^t\int_0^t
Q(\phi_u,\psi_v) V_{\epsilon,\delta}^{2H}(u-v)\ud u \ud v
\\
\label{E:IIQQeps}
&=
\frac{1}{2}\int_0^t\int_0^\theta
\left[Q(\phi_\theta,\psi_{\theta-r})+Q(\phi_{\theta-r},\psi_{\theta}) \right]V_{\epsilon,\delta}^{2H}(r)\ud r \ud \theta ,
\end{aligned}
\end{align}
where
\begin{align}\label{E:V}
V_{\epsilon,\delta}^{2H}(r)=
\frac{1}{4\epsilon\delta}\left(
 |r+\epsilon+\delta|^{2H}
+|r-\epsilon-\delta|^{2H}
-|r-\epsilon+\delta|^{2H}
-|r+\epsilon-\delta|^{2H}
\right).
\end{align}

% and
% \[
% \widetilde{Q}(\phi_t,\psi_s) := \frac{1}{2}\left[Q(\phi_t,\psi_s) + Q(\psi_t,\phi_s) \right].
% \]

\begin{lemma}\label{L:V}
There is some constant $C_H>0$ such that for all $r>0$, $\epsilon\ge \delta>0$,
\[
V_{\epsilon,\delta}^{2H}(r)1_{[4\epsilon,+\infty)}(r)\le C_H \: r^{2H-2}.
\]
\end{lemma}
\begin{proof}
% If $r<2(\epsilon+\delta)$, then $r\pm\epsilon\pm\delta\le 3(\epsilon+\delta)$. Hence,
% \begin{align*}
% V_{\epsilon,\delta}^{2H}(r) &\le \frac{3^{2H}(\epsilon+\delta)^{2H}}{\epsilon \delta}
% =3^{2H}(\epsilon+\delta)^{2H-2}
% \frac{(\epsilon+\delta)^2}{\epsilon \delta}\\
% &\le 3^{2H} 2^{2-2H} [2(\epsilon+\delta)]^{2H-2} \:8
% \le 3^{2H} 2^{5-2H}\: r^{2H-2}.
% \end{align*}
Because $r\ge 4\epsilon\ge 2(\epsilon+\delta)$, we see that $r\pm\epsilon\pm \delta>0$ and
\begin{align*}
V_{\epsilon,\delta}^{2H}(r) =&
\frac{1}{4}\int_{-1}^{1}\int_{-1}^{1}\alpha_{H} (r+\eta\epsilon+\xi\delta)^{2H-2}\ud \xi \ud \eta.
% \le &
% \frac{1}{4}\int_{-1}^{1}\int_{-1}^{1}\alpha_{H} \left(1+\frac{\eta\epsilon+\xi\delta}{r}\right)^{2H-2} r^{2H-2}\ud \xi \ud \eta\\
\end{align*}
Because
\begin{align*}
(r+\eta\epsilon+\xi\delta)^{2H-2} &=
\left(1+\frac{\eta\epsilon+\xi\delta}{r}\right)^{2H-2} r^{2H-2}
\le
\left(1-\frac{\epsilon+\delta}{r}\right)^{2H-2} r^{2H-2}\\
&\le
\left(1-\frac{1}{2}\right)^{2H-2} r^{2H-2} = 2^{2-2H} r^{2H-2},
\end{align*}
we have that $V_{\epsilon,\delta}^{2H}(r) \le 2^{3-2H}H(2H-1) r^{2H-2}$.
This proves Lemma \ref{L:V}.
\end{proof}

\begin{lemma}\label{L:Diagnal}
 If $\psi:[0,T]\mapsto \R$ is a bounded function, then
 either for $\widehat{\psi}(t,\theta)=\psi(\theta)$ or for $\widehat{\psi}(t,\theta)=\psi(t-\theta)$, we have that
\begin{align}
\left|\int_0^t\ud \theta \: \widehat{\psi}(t,\theta)\int_0^\theta\ud r\:
V_{\epsilon,\delta}^{2H}(r) - 2H \int_0^t \widehat{\psi}(t,\theta)  \theta^{2H-1}\ud \theta\right|
\le 4\Norm{\psi}_\infty (\epsilon+\delta)^{2H}.
\end{align}
\end{lemma}
\begin{proof}
The case $\widehat{\psi}(t,\theta)=\psi(\theta)$ is proved by Hu, Lu and Nualart in \cite[Lemma 3.2]{HLN12AOP}.
Their arguments can be easily extended to the case $\widehat{\psi}(t,\theta)=\psi(t-\theta)$.
\end{proof}

\begin{lemma}
For some constant $C_0>0$ and some $\alpha\in (0,1]$, \eqref{E:H1} holds if and only if
\begin{align}\label{E:Q2}
\left|Q(x,u) + Q(y,w)-Q(x,w)-Q(y,u)\right|\le C_{0} |x-y|^{\alpha}|u-w|^{\alpha},
% Q(x,x)+Q(y,y)-2Q(x,y)\le C_0 |x-y|^{2\alpha},
\end{align}
for all $x, y, w, u\in\R^d$.
\end{lemma}
\begin{proof}
Since $Q$ is a covariance function, one can find a process $\{Y_x,\: x\in\R^d \}$ such that
$Q(x,y)=E[Y_xY_y]$. Then the left-hand side  of \eqref{E:H1} is equal to $\E[(Y_x-Y_y)^2]$
and the left-hand side  of \eqref{E:Q2} is equal to $|\E[(Y_x-Y_y)(Y_u-Y_w)]|$.
With this representation, the equivalence between \eqref{E:H1} and \eqref{E:Q2} is clear.
\end{proof}

% beta
\bigskip
\begin{proof}[Proof of Theorem \ref{T:Def}]
Throughout the proof, we use $C$ to denote a generic constant which may vary from line to line.
Notice that
\begin{align}\label{E_:4Qs}
 \frac{1}{2}\left[Q(\phi_\theta,\psi_{\theta-r})+Q(\phi_{\theta-r},\psi_{\theta})\right]
 =&-\widehat{Q}(\theta, \theta-r, \phi, \psi)+\frac{1}{2}\left[ Q(\phi_\theta,\psi_{\theta})+Q(\phi_{\theta-r},\psi_{\theta-r})\right].
\end{align}
By \eqref{E:Q2} and by the H\"older continuity of $\phi$ and $\psi$,  we see that
\begin{align}\label{E_:Qcr}
\left|\widehat{Q}(\theta, \theta-r, \phi, \psi)\right| \le
\frac{C_0}{2} |\phi_{\theta}-\phi_{\theta-r}|^{\alpha}|\psi_{\theta}-\psi_{\theta-r}|^{\alpha}
\le \frac{C_0}{2} \Norm{\phi}_\kappa^{\alpha}\Norm{\psi}_\kappa^{\alpha}\: r^{2\alpha\kappa},
\end{align}
for all $0\le r\le \theta\le T$. Hence, using \eqref{E:IIQQeps} and \eqref{E_:4Qs},
\begin{align*}
\Bigg|\E\left[I_\epsilon(\phi)I_\delta(\psi)\right]+ &\alpha_{H}\!\int_0^t\int_0^\theta
r^{2H-2} \widehat{Q}(\theta, \theta-r, \phi, \psi) \ud r \ud \theta \\
& -
H \int_0^t \theta^{2H-1} \left[Q(\phi_\theta,\psi_{\theta})+Q(\phi_{t-\theta},\psi_{t-\theta})\right] \ud \theta\Bigg|\\
\le&
% \phantom{+\frac{1}{2} }
% \quad
\hspace{1.92em}
\left|
\int_0^t\int_0^\theta \widehat{Q}(\theta, \theta-r, \phi, \psi)\left(V_{\epsilon,\delta}^{2H}(r)-\alpha_{H}r^{2H-2}\right)\ud r \ud \theta
\right|\\
&+\frac{1}{2}\left|
\int_0^t\int_0^\theta Q(\phi_\theta,\psi_{\theta}) V_{\epsilon,\delta}^{2H}(r)\ud r \ud \theta -2H \int_0^t  \theta^{2H-1}Q(\phi_\theta,\psi_{\theta})\ud \theta
\right|\\
&+\frac{1}{2}\left|
\int_0^t\int_0^\theta Q(\phi_{\theta-r},\psi_{\theta-r}) V_{\epsilon,\delta}^{2H}(r)\ud r \ud \theta -2H \int_0^t  \theta^{2H-1}Q(\phi_{t-\theta},\psi_{t-\theta})\ud \theta
\right|\\
=&:\hspace{0.5em}I_1 + \frac{I_2}{2} +\frac{I_3}{2}.
\end{align*}
We claim that
\begin{align}\label{E:limI123}
\lim_{\epsilon,\delta \rightarrow 0} I_i =0 ,\quad i=1,2,3.
\end{align}
Therefore, we have that
\begin{align}
\label{E_:SecM}
\begin{aligned}
\lim_{\epsilon,\delta \rightarrow 0} \E\left[I_\epsilon(\phi)I_\delta(\psi)\right]
=& H\int_0^t \theta^{2H-1}\left[Q(\phi_{\theta},\psi_{\theta})+Q(\phi_{t-\theta},\psi_{t-\theta})\right]
\ud \theta\\
&-\alpha_{H}\int_0^t\int_0^\theta
r^{2H-2} \widehat{Q}(\theta, \theta-r, \phi, \psi) \ud r \ud \theta.
\end{aligned}
\end{align}
When $\psi=\phi$, this implies that $\{I_{\epsilon_n}(\phi),\: n\ge 1\}$ is a Cauchy sequence in $L^2(\Omega)$ for any
sequence $\epsilon_n\downarrow 0$.
Therefore, $\lim_{\epsilon\rightarrow 0} I_\epsilon(\phi)$ exists in $L^2(\Omega)$ and is denoted by
$I(\phi):=\int_0^t W(\ud s,\phi_s)$.
Formula \eqref{E:SecM} is a consequence of \eqref{E_:SecM}.
% By symmetry, we have \eqref{E:SecM-Sym}.
As for moment bound \eqref{E:BddM}, by \eqref{E:H3} and \eqref{E_:Qcr},
\begin{align*}
\left|\E[I(\phi)I(\psi)]\right|
\le &
\quad \frac{|\alpha_H|C_0}{2} \Norm{\phi}_\kappa^\alpha \Norm{\psi}_\kappa^\alpha \int_0^t\int_0^\theta r^{2\alpha\kappa+2H-2} \ud r\ud \theta \\
&+C_1 \left(1+\Norm{\phi}_\infty\vee \Norm{\psi}_\infty\right)^{2\alpha} (2H)\int_0^t \theta^{2H-1} \ud \theta \\
=&  \frac{C_0|\alpha_H|\Norm{\phi}_\kappa^{\alpha}\Norm{\psi}_\kappa^{\alpha} t^{2(H+\alpha\kappa)}}{4(H+\alpha\kappa)(2(H+\alpha\kappa)-1)}
+ C_1(1+\Norm{\phi}_\infty\vee \Norm{\psi}_\infty)^{2\alpha} t^{2H}.
\end{align*}
Therefore, it remains to prove \eqref{E:limI123}, which will be done in the following two steps.

\paragraph{Step 1.}
We first prove \eqref{E:limI123} for $I_1$. Notice that $I_1$ can be decomposed as
\begin{align*}
I_1 \le& \hspace{1.05em} \left|
\int_{0}^t\int_{4\epsilon}^\theta \widehat{Q}(\theta, \theta-r, \phi, \psi)
\left(V_{\epsilon,\delta}^{2H}(r)-\alpha_{H}r^{2H-2}\right)
\ud r \ud \theta
\right|\\
& +
\left|
\int_{0}^t\int_0^{4\epsilon}\widehat{Q}(\theta, \theta-r, \phi, \psi)
\left(V_{\epsilon,\delta}^{2H}(r)-\alpha_{H}r^{2H-2}\right)\ud r \ud \theta
\right|\\
=:&\hspace{0.5em}  I_{1,1} + I_{1,2}.
\end{align*}
Notice that for $r>0$,
\[
\lim_{\epsilon,\delta\rightarrow 0}V_{\epsilon,\delta}^{2H}(r) = \alpha_{H} r^{2H-2}.
\]
Because $H+\alpha\kappa>1/2$, by Lemma \ref{L:V} and \eqref{E_:Qcr}, we can apply dominated convergence theorem to see that
\[
\lim_{\epsilon,\delta\rightarrow 0} I_{1,1} =0.
\]
As for $I_{1,2}$, we see that
\begin{align*}
I_{1,2}=&\left|\int_0^t\ud \theta\int_0^{4\epsilon} \widehat{Q}(\theta, \theta-r, \phi, \psi) \left(V_{\epsilon,\delta}^{2H}(r)-\alpha_{H}r^{2H-2}\right)\ud r\right|\\
\le&\frac{C}{\epsilon}\int_0^t\ud \theta\int_0^{4\epsilon} \ud r\: \left|\widehat{Q}(\theta, \theta-r, \phi, \psi)\right|\\
&\times \int_{-1}^1 \ud y\: \left[|r-\epsilon+\delta y|^{2H-1}+|r+\epsilon+\delta y|^{2H-1}+\epsilon r^{2H-2}\right].
\end{align*}
Then by \eqref{E_:Qcr},
\begin{align*}
I_{1,2} &\le \frac{C\Norm{\phi}_\kappa^{\alpha}\Norm{\psi}_\kappa^{\alpha} t}{\epsilon}
\int_0^{4\epsilon}\ud r\: r^{2\alpha\kappa}\int_{-1}^1 \ud y\: \left[|r-\epsilon+\delta y|^{2H-1}+|r+\epsilon+\delta y|^{2H-1}+ \epsilon r^{2H-2}\right]\\
&\le
C\Norm{\phi}_\kappa^{\alpha} \Norm{\psi}_\kappa^{\alpha} t\:  \epsilon^{2\alpha\kappa-1}
\int_{-1}^1 \ud y \int_0^{4\epsilon} \ud r \left[|r-\epsilon+\delta y|^{2H-1}+|r+\epsilon+\delta y|^{2H-1}+\epsilon r^{2H-2}\right].
\end{align*}
Because $\epsilon>\delta>0$ and $y\in [-1,1]$, we have that
\begin{align*}
\int_0^{4\epsilon} \ud r \left[|r+\epsilon+\delta y|^{2H-1}+\epsilon r^{2H-2}\right]
&=\int_0^{4\epsilon} \ud r \left[(r+\epsilon+\delta y)^{2H-1}+\epsilon r^{2H-2}\right]\\
&\le C [(5\epsilon+\delta y)^{2H} + \epsilon^{2H}]\le C\: \epsilon^{2H};
\end{align*}
and because $\epsilon-\delta y\in [0,4\epsilon]$, we see that
\begin{align*}
\int_0^{4\epsilon} \ud r |r-\epsilon+\delta y|^{2H-1}
&=
\int_0^{\epsilon-\delta y} \ud r \: (\epsilon-\delta y-r)^{2H-1} +
\int_{\epsilon-\delta y}^{4\epsilon} \ud r \: (r-\epsilon+\delta y)^{2H-1} \\
&=\frac{1}{2H} \left[(\epsilon-\delta y)^{2H}+(3\epsilon+\delta y)^{2H}\right]\le C \: \epsilon^{2H}.
\end{align*}
Hence,
\[
I_{1,2}\le C\Norm{\phi}_\kappa^{\alpha} \Norm{\psi}_\kappa^{\alpha} \: t\:  \epsilon^{2\alpha\kappa-1+2H}.
\]
Therefore, the condition $\alpha\kappa+H\ge 1/2$ implies
\[
\lim_{\epsilon,\delta\rightarrow 0} I_{1,2} =0.
\]

\paragraph{Step 2.}
Now we prove \eqref{E:limI123} for $I_2$ and $I_3$.
The case for $I_2$ is true due to Lemma \ref{L:Diagnal}. As for $I_3$, notice that
\begin{align*}
\int_0^t\ud \theta \int_0^\theta \ud r \:  Q(\phi_{\theta-r},\psi_{\theta-r})V_{\epsilon,\delta}^{2H}(r) &=
\int_0^t\ud r\: V_{\epsilon,\delta}^{2H}(r)\int_r^t \ud \theta \: Q(\phi_{\theta-r},\psi_{\theta-r})\\
&=
\int_0^t\ud r\: V_{\epsilon,\delta}^{2H}(r) \int_0^{t-r} \ud s \: Q(\phi_{s},\psi_{s})
\\
&=
\int_0^t\ud s\: Q(\phi_{s},\psi_{s}) \int_0^{t-s}  \ud r \:V_{\epsilon,\delta}^{2H}(r)
\\
&=
\int_0^t\ud \theta\: Q(\phi_{t-\theta},\psi_{t-\theta})\int_0^\theta \ud r \: V_{\epsilon,\delta}^{2H}(r)
.
\end{align*}
Hence, one can apply Lemma \ref{L:Diagnal} to prove \eqref{E:limI123} for $I_3$.
This completes the proof of Theorem \ref{T:Def}.
\end{proof}

\begin{proof}[Proof of Proposition \ref{P:Holder}]
We only need to prove that
\begin{align}\label{E:Holder-e}
\begin{aligned}
 \E\left[\left(\int_0^t \W^\epsilon(\ud r,\phi_r)-\int_0^s \W^\epsilon(\ud r,\phi_r)\right)^2\right]
 \le & \quad C' \left(1+\Norm{\phi}_\infty\right)^{2\alpha} (t-s)^{2H}\\
 & + C'' \Norm{\phi}_\kappa^{2\alpha} (t-s)^{2(H+\alpha\kappa)}.
\end{aligned}
\end{align}
Then \eqref{E:Holder} follows from \eqref{E:Holder-e}, Theorem \ref{T:Def}, and Fatou's lemma.
By the arguments in the proof of \cite[Proposition 3.6]{HLN12AOP} and by denoting $\hat{\phi}_t=\phi_{t+s}$, we see that
\begin{align*}
\E\left[\left(\int_0^t \W^\epsilon(\ud r,\phi_r)-\int_0^s \W^\epsilon(\ud r,\phi_r)\right)^2\right]
&=\int_0^{t-s}\ud \theta \int_0^\theta\ud r\:
Q(\phi_{s+\theta},\phi_{s+\theta-r}) V_{\epsilon,\epsilon}^{2H}(r)\\
&=\int_0^{t-s}\ud \theta \int_0^\theta\ud r\:
Q(\hat{\phi}_{\theta},\hat{\phi}_{\theta-r}) V_{\epsilon,\epsilon}^{2H}(r)\\
&=
\int_0^{t-s} \int_0^{t-s}
Q(\hat{\phi}_{u},\hat{\phi}_{v}) V_{\epsilon,\epsilon}^{2H}(u-v)
\ud u\ud v\\
&=\E\left[I_\epsilon^2(\hat{\phi})\right],
\end{align*}
where the last equality is due to \eqref{E:IIQQeps}.
Finally, after passing the limit using \eqref{E_:SecM} and then applying the bound in \eqref{E:BddM},
we complete the proof of Proposition \ref{P:Holder}.
\end{proof}

\section{Feynman-Kac formula and upper bound of moments}\label{S:Upper}

In this section, we will establish the
Feynman-Kac representation of the solution to \eqref{E:SHE} and
and obtain a upper bound of its moments.

\subsection{Feynman-Kac integral and its moment bound}
The goal of this part is to prove the upper bound in \eqref{E:uplowBd}.

\begin{theorem}\label{T:ExpInt}
Suppose that $Q$ satisfies condition \eqref{E:H1} with $2H+\alpha>1$ and $u_0$ is bounded.
Then for all $t>0$ and $x\in\R^d$, the random variable $\int_0^t W(\ud s,B_{t-s}^x)$
is exponentially integrable and the random field $u(t,x)$ given by \eqref{E:FKF}
is in $L^p(\Omega)$ for all $p\ge 1$.
Moreover,  for some constant
$C=C(d,H,\alpha,u_0)>0$,
\[
\E\left[|u(t,x)|^k\right]\le C\exp\left(C  k^{\frac{2-\alpha}{1-\alpha}} \:
t^{\frac{2H+\alpha}{1-\alpha}} + C k^{\frac{2-\alpha}{1-\alpha}} t^{\frac{2H+\alpha}{1-\alpha}}\right),
\]
for all $t\ge 1$ and $x\in\R^d$.
\end{theorem}

We first prove some lemmas.

\begin{lemma}\label{L:MomGen}
Suppose that $\alpha\in (0,1]$ and $2H+\alpha>1$.
Let
\begin{align}\label{E:U}
 U=\int_0^1\int_0^1 |B_{u}-B_{v}|^{2\alpha}
|u-v|^{2H-2}\ud u\ud v,
\end{align}
where  $B_t$ is the standard Brownian motion on $\R^d$.
 Then for some constant $C_{\alpha,d,H}>0 $,
\[
\E\left[e^{\lambda U}\right]\le C_{\alpha,d,H} \exp\left(C_{\alpha,d,H}  \lambda^{\frac{1}{1-\alpha}} \right), \quad \text{for all $\lambda\ge 0$.}
\]
\end{lemma}
\begin{proof}
Notice that
\[
\E\left[|B_{u}-B_{v}|^{2\alpha n}\right]=
\E\left[|B_{u-v}|^{2\alpha n}\right]=
|u-v|^{\alpha n}
\E\left[|B_1|^{2\alpha n}\right]
=C_d 2^{\alpha n} \Gamma(d/2+n\alpha)|u-v|^{\alpha n}.
\]
By Minkovski's inequality,
\begin{align*}
\E[U^n]
&\le
\left[\int_0^1\int_0^1 \E\left[|B_{u}-B_{v}|^{2\alpha n}\right]^{1/n}
|u-v|^{2H-2}\ud u\ud v\right]^n\\
&=
\left[\int_0^1\int_0^1 \left[C_d 2^{\alpha n} |u-v|^{\alpha n}\Gamma(d/2+\alpha n)\right]^{1/n}
|u-v|^{2H-2}\ud u\ud v\right]^n\\
&=
C_d  2^{\alpha n}\Theta^n   \Gamma(d/2+\alpha n),
\end{align*}
where
\[
\Theta:=\int_0^1\int_0^1
|u-v|^{2H-2+\alpha}\ud u\ud v
=\frac{2}{(2 H+\alpha-1) (2 H+\alpha )}.
\]
By Lemma \ref{L:Gamma}, we see that for some constant $C_{\alpha,d}>0$,
\[
\frac{\Gamma(d/2+\alpha n)}{n!} \le \frac{C_{\alpha,d}}{\Gamma((3-d)/2+(1-\alpha) n)}, \quad\text{for all $n\in\bbN$.}
\]
Notice that $d\ge 1$ implies that $(3-d)/2\le 1$. Hence, by Lemma \ref{L:ML-bds},
we see that for some constant $C_{\alpha,d,H}\ge 1$,
\begin{align*}
 \E[e^{\lambda U}] &=\sum_{n=0}^\infty \frac{\lambda^n}{n!} \E[U^n]\le
\sum_{n=0}^\infty \frac{\lambda^n}{n!} C_{d}\: 2^{\alpha n}\Theta^n  \Gamma(d/2+\alpha n)\\
&\le
C_{d} \sum_{n=0}^\infty \frac{\lambda^n2^{\alpha n}\Theta^n }{\Gamma((3-d)/2+(1-\alpha) n)}\\
&\le C_{\alpha,d, H}\exp\left(C_{\alpha,d, H} \lambda^{\frac{1}{1-\alpha}} [2^\alpha\Theta]^{\frac{1}{1-\alpha}} \right),
\end{align*}
for $\lambda\ge 0$.
This proves Lemma \ref{L:MomGen}.
\end{proof}

\begin{lemma}\label{L:RunMax}
Let $B_t$ be a standard Brownian motion on $\R^d$ and $W = \sup_{s\in[0,1]} |B_s|$.
Then for all $M\in [0,2)$, there exists some constant $C_{M,d}>0$ such that
\[
\E\left[e^{\lambda (1+ W)^{2\alpha} }\right]\le C_{M,d} \exp\left(C_{M,d}\: \lambda^{\frac{2}{2-M}}\right),\quad\text{for all $\lambda\ge 0$.}
\]
\end{lemma}
\begin{proof}
By Fernique's theorem, for some $\alpha_d>0$ it holds that
\[
\E\exp\left(\alpha W^2\right)<\infty\quad\text{for all $\alpha<\alpha_d$.}
\]
Apply the inequality $ab\le p^{-1} a^p+q^{-1}b^q$ where $a,b\ge0$ and $1/p+1/q=1$ to see that
\[
\E\left[e^{\lambda (1+ W)^{2\alpha} }\right]\le
\E\left[e^{p^{-1}\left(\frac{\lambda}{a}\right)^p + q^{-1}a^q(1+ W)^{Mq} }\right].
\]
Then the lemma is proved by choosing $q=\frac{2}{M}$, $p=\frac{2}{2-M}$ and $a$ sufficiently small such that $q^{-1}a^q<\alpha_d$.
\end{proof}
% beta

\bigskip
\begin{proof}[Proof of Theorem \ref{T:ExpInt}]
Let $u(t,x)$ be the random field given by \eqref{E:FKF}.
Without of loss of generality, we may assume that $u_0(x)\equiv 1$.
Notice that
\begin{align}\notag
\E\left[u(t,x)^k\right]&= \E^W\E^B\exp\left\{\sum_{j=1}^k\int_0^t W(\ud s,B_{t-s}^{j,x})\right\}\\
\notag
&= \E^B \exp\left\{\frac{1}{2}\E^W\left[\left|\sum_{j=1}^k\int_0^t W(\ud s,B_{t-s}^{j,x})\right|^2 \right]\right\}\\
\label{E:u^k}
&= \E^B\exp\left\{
\frac{1}{2}\sum_{i,j=1}^k \E^W\left[\int_0^t W(\ud s,B_{t-s}^{i,x})\int_0^t W(\ud s,B_{t-s}^{j,x}) \right]
\right\},
\end{align}
where $\{B_t^{j,x},t\ge 0\}$, $1\le j\le k$, are independent Brownian motions on $\R^d$ starting from $x$.
By \eqref{E:BddM-Int},
\begin{align*}
 \E\left[u(t,x)^k\right]
&\le
\E^B\exp\Bigg\{
\sum_{i,j=1}^k \frac{C_0 |\alpha_H|}{2} \int_0^t\int_0^t |B_u^{i}-B_v^{i}|^\alpha|B_u^{j}-B_v^{j}|^\alpha |u-v|^{2H-2}\ud u\ud v \\
\notag
&\qquad +
H \sum_{i,j=1}^k \int_0^t \theta^{2H-1}\left[Q(B_\theta^{i},B_\theta^{j})+Q(B_{t-\theta}^{i},B_{t-\theta}^{j})\right]\ud \theta
\Bigg\}.
\end{align*}
Then by Cauchy-Schwartz inequality,
\begin{align*}
\E\left[u(t,x)^k\right]\le&
\E^B\left[\exp\left\{
\sum_{i,j=1}^k C_0|\alpha_H| \int_0^t\int_0^t |B_u^{i,x}-B_v^{i,x}|^\alpha|B_u^{j,x}-B_v^{j,x}|^\alpha |u-v|^{2H-2}\ud u\ud v \right\}\right]^{1/2}\\
&\times \E^B\left[\exp\left\{
2H\sum_{i,j=1}^k \int_0^t \theta^{2H-1}\left[Q(B_\theta^{i,x},B_\theta^{j,x})+Q(B_{t-\theta}^{i,x},B_{t-\theta}^{j,x})\right]\ud \theta
\right\}\right]^{1/2}\\
=:
&
\left( \E^B\left[I_1\right]\E^B\left[I_2\right]\right)^{1/2}.
\end{align*}

\paragraph{Step 1.} We first consider $\E^B\left[I_1\right]$:
\begin{align*}
\E^B\left[I_1\right] \le&\:
\E^B\exp\left\{
\sum_{i,j=1}^k \frac{C_0|\alpha_H|}{2} \int_0^t\int_0^t \left[|B_u^{i}-B_v^{i}|^{2\alpha} +
|B_u^{j}-B_v^{j}|^{2\alpha} \right] |u-v|^{2H-2}\ud u \ud v  \right\}\\
=&
\E^B\exp\left\{
C_0 k\sum_{i=1}^k |\alpha_H| \int_0^t\int_0^t |B_u^{i}-B_v^{i}|^{2\alpha} |u-v|^{2H-2}\ud u \ud v\right\}\\
=&
\left[\E^B\exp\left(
C_0 k|\alpha_H|  \int_0^t\int_0^t |B_u-B_v|^{2\alpha} |u-v|^{2H-2}\ud u \ud v \right)\right]^k.
\end{align*}
By change of variables $u=t u'$ and $v=tv'$ and by the scaling property of Brownian motions,
\begin{eqnarray*}
 \int_0^t\int_0^t |B_u-B_v|^{2\alpha}|u-v|^{2H-2}\ud u\ud v
&=&t^{2H}\int_0^1\int_0^1 |B_{tu'}-B_{tv'}|^{2\alpha}|u'-v'|^{2H-2}\ud u'\ud v'\\
&\stackrel{\text{in law}}{=}&t^{2H+\alpha}\int_0^1\int_0^1 |B_{u'}-B_{v'}|^{2\alpha}
|u'-v'|^{2H-2}\ud u'\ud v'.
\end{eqnarray*}
Hence,
\begin{align*}
\E^B\exp\left(
C_0 k|\alpha_H|  \int_0^t\int_0^t |B_u-B_v|^{2\alpha} |u-v|^{2H-2}\ud u \ud v \right)
=
\E^B\exp\left(
C_0 k|\alpha_H|  \: t^{2H+\alpha} U \right),
\end{align*}
where $U$ is defined in \eqref{E:U}. Then apply Lemma \ref{L:MomGen} to $\E^B\exp\left(
C_0 k|\alpha_H|  \: t^{2H+\alpha} U \right)$ to see that for some constant $C_{\alpha,d,H}>0$,
\[
\E^B[I_1]^{1/2}\le C_{\alpha,d,H}\exp\left(C_{\alpha,d,H} k^{\frac{2-\alpha}{1-\alpha}} t^{\frac{2H+\alpha}{1-\alpha}}\right), \quad\text{for all $t\ge 0$.}
\]

\paragraph{Step 2.}
Now we study $E^B[I_2]$. Set $\Norm{B}_{\infty, t}=\sup_{0\le s\le t}|B_s|$. By condition \eqref{E:H3},
\begin{align*}
 \E[I_2 ]&\le\E^B \exp\left(2H C_1 \sum_{i,j=1}^k
\left[\left(1+\Norm{B^i}_{\infty,t}\right)^{2\alpha}+
\left(1+\Norm{B^j}_{\infty,t}\right)^{2\alpha}\:\right]
\int_0^t \theta^{2H-1}\ud\theta \right)\\
&\le \E^B \exp\left(C_1  k \: t^{2H}  \sum_{i=1}^k
\left(1+\Norm{B^i}_{\infty,t}\right)^{2\alpha}
\right)\\
&= \left[\E^B \exp\left( C_1 k \: t^{2H}
\left(1+\Norm{B^1}_{\infty,t}\right)^{2\alpha}
\right)\right]^k.
\end{align*}
By scaling property and Lemma \ref{L:RunMax}, we see that for some constant $C_{\alpha,d}'>0$,
\begin{align*}
\E^B \exp\left( C_1 k\: t^{2H}
\left(1+\Norm{B^1}_{\infty,t}\right)^{2\alpha}
\right)
&\le \E^B \exp\left( C_1 k \: t^{2H} (t\vee 1)^{\alpha}
\left(1+\Norm{B^1}_{\infty,1}\right)^{2\alpha}
\right)\\
&\le C_{\alpha,d}'\exp\left(C_{\alpha,d}' \: k^{\frac{1}{1-\alpha}} \: \left[t^{2H} (t\vee 1)^{\alpha}\right]^{\frac{1}{1-\alpha}}\right)
\end{align*}
Hence, for some constant $C_{\alpha,d}>0$,
\[
\E^B[I_2]^{1/2}\le C_{\alpha,d}\exp\left(C_{\alpha,d}
k^{\frac{2-\alpha}{1-\alpha}} \:
t^{\frac{2H+\alpha}{1-\alpha}} \right)
\qquad\text{for all $t\ge 1$.}
\]
Finally, Theorem \ref{T:ExpInt} is proved by combining the results in the above two steps.
\end{proof}

\subsection{Validation of the Feynman-Kac formula}
\label{S:Validation}

In this part we will show that $u(t,x)$ is a weak solution to \eqref{E:SHE}.

\begin{definition}\label{D:Stratonovich}
 Given a random field $v=\{v(t,x),\: t\ge 0,\: x\in\R^d \}$ such that
\[
\int_0^t\int_{\R^d}|v(s,x)|\ud x\ud s<\infty\quad\text{ a.s. for all $t>0$,}
\]
the {\it Stratonovich integral} is defined as the following limit in probability if it exists
\[
\lim_{\epsilon\rightarrow 0}\int_0^t \int_{\R^d}v(s,x)\W^\epsilon(s,x)\ud s\ud x,
\]
where $\W^\epsilon(t,x)$ is defined in \eqref{E:WEpsilon}.
\end{definition}

\begin{definition}\label{D:Weak}
A random field $u=\{u(t,x),\: t\ge 0, x\in\R^d\}$ is a {\it weak solution} to \eqref{E:SHE} if for any
$\phi\in C_0^\infty(\R^d)$, we have that
\begin{align}
\begin{aligned}
 \int_{\R^d}\left[u(t,x)-u_0(x)\right]\phi(x)\ud x =
 &\quad \int_0^t\int_{\R^d} u(s,x) \Delta \phi(x) \ud x\ud s\\
 &+\int_0^t\int_{\R^d} u(s,x) \phi(x) W(\ud s,x) \ud x,
\end{aligned}
\end{align}
almost surely, for all $t>0$, where the last term is a Stratonovich stochastic integral defined in \ref{D:Stratonovich}.
\end{definition}

\begin{theorem}\label{T:Weak}
 Suppose that $Q$ satisfies condition \eqref{E:H1} with $2 H + \alpha>1$ and $u_0$ is a bounded measurable function.
 Let $u(t,x)$ be the random field defined in \eqref{E:FKF}. Then for any $\phi\in C_0^\infty(\R^d)$,
 $u(t,x)\phi(x)$ is Stratonovich integrable and $u(t,x)$ is a weak solution to \eqref{E:SHE} in the sense
 of Definition \ref{D:Weak}.
\end{theorem}
% \begin{proof}
With Theorems \ref{T:Def} and \ref{T:ExpInt}, and Proposition \ref{P:Holder},
the proof of Theorem \ref{T:Weak} follows exactly the same arguments
as those of Theorem 5.3 in \cite{HLN12AOP}. We will not repeat the proofs and instead leave them to
interested readers.
% \end{proof}

\section{Lower bounds of moments}
\label{S:Low}

In this section, we prove the lower bound in \eqref{E:uplowBd}.

\begin{theorem}\label{T:LowBd}
Suppose that $Q$ satisfies condition \eqref{E:H1} with $2H+\alpha>1$ and $\inf_{x\in\R^d} u_0>0$.
If $Q$ satisfies condition \eqref{E:H2} as well for some $\beta\in [0,1)$,
then there exists some $C=C(d,H,\alpha,\beta,u_0)>0$ such that for all $x\in\R^d$,
if either $k$ or $t$ is sufficiently large, then
\[
\E\left[u(t,x)^k\right]\ge C\exp\left(C k^{\frac{2-\beta}{1-\beta}} t^{\frac{2H+\beta}{1-\beta}}\right).
% \qquad\text{for all $(t,x)\in\R_+\times\R^d$.}
\]
\end{theorem}

We first remark that if the initial data is $u_0(x)\equiv 1$,
then from \eqref{E:u^k} and \eqref{E:Formal}, we see that
\begin{align*}
\E\left[u(t,x)^k\right]
=\E^B\exp\Bigg\{
&\frac{|\alpha_H|}{2}\sum_{i,j=1}^k \int_0^t\int_0^t |u-v|^{2H-2}\widehat{Q}(u,v,B^{i,x},B^{j,x})\ud u\ud v\\
&+H\sum_{i,j=1}^k \int_0^t Q(B_s^{i,x},B_s^{i,x})\left[s^{2H-1}+(t-s)^{2H-1}\right]\ud s
\Bigg\}.
\end{align*}
Since the sign of $\widehat{Q}$ can be either positive or negative, it is hard to find a lower bound starting from the above formula.
Instead, we will introduce another Gaussian field $Y$ as in Lemma \ref{L:WWYY} below.

Now we need some notation. Fix $a>0$. Let $\kappa=H-1/2$.
As is proved in \cite{P-Taqqu01}, the space
\begin{align}\label{E:Ha}
 \calH_a = \left\{
f:\: \exists \phi_f\in L^2(0,a) \:\: \text{such that }\:f(u)= u^{-\kappa}
\left(I_{a-}^{-\kappa} \phi_f(s)\right)(u)
\right\}
\end{align}
with the inner product
\begin{align*}
\InPrd{f,g}_{\calH_a}
&=\frac{\pi\kappa (2\kappa+1)}{\Gamma(1-2\kappa)\sin(\pi\kappa)}
\int_0^a s^{-2\kappa}\left(I_{a-}^\kappa u^\kappa f(u)\right)\!\!(s)
\left(I_{a-}^\kappa u^\kappa g(u)\right)\!\!(s) \:\ud s
% &=\frac{\pi\kappa (2\kappa+1)}{\Gamma(1-2\kappa)\sin(\pi\kappa)}
% \int_0^a\phi_f(s) \phi_g(s)\ud s,
\end{align*}
is a Hilbert space, where $I_{a-}^\kappa$ with $\kappa<0$ is the right-sided fractional derivative (see \cite{P-Taqqu01}).
It is known that (see \cite[p. 284]{Nualart06})
\begin{align}\label{E:Spaces}
C^{\gamma}([0,a]) \subset \calH_a \subset L^2(0,a),\quad
\text{for all $\gamma>1/2-H$.}
\end{align}

\begin{lemma}\label{L:WWYY}
There exist a Gaussian process $Y=\{Y(x), x\in\R^d\}$ and an independent fractional Brownian motion
$\{\widehat{B}_t,t\in \R\}$ with Hurst parameter $H$, such that
\begin{enumerate}[(a)]
\item For all $x,y\in\R^d$, $\E[Y(x)Y(y)] = Q(x,y)$.
\item For all $0<t\le T$ and $\phi\in  C^\kappa([0,T])$ with $\alpha\kappa+H>1/2$,
the integral $\int_0^t Y(\phi_s)\ud \widehat{B}_s$ is a well-defined Wiener integral
for each realization of $Y$.
Moreover,
\begin{align}\label{E:YL2}
\int_0^t Y(\phi_s)\ud \widehat{B}_s = \lim_{\epsilon\rightarrow 0}\frac{1}{2\epsilon}\int_0^t
Y(\phi_s)\left(\widehat{B}_{s+\epsilon}-\widehat{B}_{s-\epsilon}\right)\ud s,\quad\text{in $L^2(\Omega)$.}
\end{align}
\item For all $0<t\le T$ and $\phi,\psi\in  C^\kappa([0,T])$ with $\alpha\kappa+H>1/2$,
\begin{align}\label{E:WWYY}
\E^W\left[\int_0^tW(\ud s,\phi_s)\int_0^tW(\ud s,\psi_s)\right]
=\E^{Y,\widehat{B}}\left[\int_0^t Y(\phi_s)\ud \widehat{B}_s\int_0^t Y(\psi_s)\ud \widehat{B}_s\right].
\end{align}
\end{enumerate}
\end{lemma}
\begin{proof}
Since $Q$ is a covariance function, one can find such a Gaussian process $Y$ such that part (a) holds.
As for (b), by \eqref{E:H1}, we see that
\begin{align}
\label{E:Holder-Y}
\begin{aligned}
\E\left[\left|Y(\phi_t)-Y(\phi_s)\right|^p\right] &\le
C_p\E\left[\left|Y(\phi_t)-Y(\phi_s)\right|^2\right]^{p/2}\\
&\le C_p' \left|\phi_t-\phi_s\right|^{\alpha p}\le
C_p' \Norm{\phi}_\kappa \left|t-s\right|^{\alpha \kappa p}.
\end{aligned}
\end{align}
Hence, $t\mapsto Y(\phi_t)$ is $\gamma$-H\"older continuous for all $\gamma<\alpha\kappa$.
Since $\alpha\kappa>1/2-H$, one can find $\gamma'$ such that $1/2-H<\gamma'<\alpha\kappa$.
Because $Y$ and $\widehat{B}$ are independent, for each realization of $Y$,
the integral $\int_0^tY(\phi_s)\ud \widehat{B}_s$ is actually a Wiener integral.
By \eqref{E:Spaces}, we see that the integral $\int_0^tY(\phi_s)\ud \widehat{B}_s$ is a well-defined
Wiener integral for each realization of $Y$.

% \bigskip
As for \eqref{E:YL2}, denote
\[
I_\epsilon(\phi):=\frac{1}{2\epsilon}\int_0^t
Y(\phi_s)\left(\widehat{B}_{s+\epsilon}-\widehat{B}_{s-\epsilon}\right)\ud s.
\]
Then by the same arguments as the proof of Theorem \ref{T:Def}, one can show that $I_\epsilon(\phi)$ is
a Cauchy sequence in $L^2(\Omega)$. Denote the limit by $I(\phi)$.
In order to show that $I(\phi)$ equals to the left-hand side of
 (\ref{E:YL2}),  it suffices to show that  for any $t_0 \in [0,t]$ and any bounded random variable $Z$ measurable with respect to the process $Y$, we have that
\begin{equation}  \label{E1}
 \E \left[ \widehat{B}_{t_0} Z \int_0^t Y(\phi_s) \ud \widehat{B}_s \right] = \E \left[ \widehat{B}_{t_0} Z I(\phi) \right].
\end{equation}
For the  right-hand side of \eqref{E1}, we can write
\begin{align}
  \E \left[ \widehat{B}_{t_0} Z I(\phi) \right] =&  \lim_{\epsilon\rightarrow 0}\frac{1}{2\epsilon}\int_0^t \E[Z Y(\phi_s)]  (R_H(t_0, s+\epsilon)- R_H(t_0, s-\epsilon)) \ud s \notag\\
=&2H \int_0^t \E[Z Y(\phi_s)]    \label{E2}
  ( s^{2H-1} +|t_0-s| ^{2H-1} {\rm sign} (t_0-s))\ud s.
\end{align}
On the other hand, by Fubini's theorem,  the left-hand side of \eqref{E1} equals to
 \[
  \E^{\widehat{B}} \left[ \widehat{B}_{t_0}  \int_0^t  \E^{Y}[Z Y(\phi_s)] \ud \widehat{B}_s \right],
\]
which coincides with  \eqref{E2}, due to the properties of stochastic $Y$-integrals.
In fact, this property holds when  $\E[Z Y(\phi_s)] $ is a step function and it holds for any element
in space $\mathcal{H}_t$ (see \eqref{E:Ha}) of integrable functions on $[0,t]$, because
 $\mathcal{H}_t$ is continuously embedded into $L^{1/H} (0,t)$.
\bigskip

(c) Because $Y$ and $\widehat{B}$ are independent, by \eqref{E:YL2}, we see that
\begin{align}
\notag
 \E^{Y,\widehat{B}}\Bigg[\int_0^t Y(\phi_s)&\ud \widehat{B}_s\int_0^t Y(\psi_s)\ud \widehat{B}_s\Bigg]\\
 \notag
 &=
 \lim_{\epsilon\rightarrow 0}\E^{Y,\widehat{B}}\left[\int_0^t Y(\phi_s)
 \frac{\widehat{B}_{s+\epsilon}-\widehat{B}_{s-\epsilon}}{2\epsilon}\ud s
 \int_0^t Y(\psi_s)\frac{\widehat{B}_{s+\epsilon}-\widehat{B}_{s-\epsilon}}{2\epsilon}\ud s\right]\\
 &=
 \lim_{\epsilon\rightarrow 0}
 \int_0^t\int_0^t Q\left(\phi_u,\psi_v\right) V_{\epsilon,\epsilon}^{2H}(u-v) \ud u\ud v,
 \label{E:LimitQV}
\end{align}
where $V_{\epsilon,\delta}^{2H}(\cdot)$ is defined in \eqref{E:V}.
The limit in \eqref{E:LimitQV} has been calculated in Theorem \ref{T:Def} and it is
equal to the right-hand side of \eqref{E:SecM} or \eqref{E:Formal}.
This completes the proof of Lemma \ref{L:WWYY}.
\end{proof}

% \begin{lemma}[Theorem 2.1 (i) in \cite{MMV00}]
% Assume that $\{\widehat{B}_s\}_{s\ge 0}$ is a fractional Brownian motion with $H\in(0,1/2)$.
% Then there exists a constant $\theta:=\theta(H,r)$ such that for all $a>0$ and $r>0$, it
% holds that
%
% \end{lemma}

\begin{lemma}\label{L:LowBd}
 Assume that $\{\widehat{B}_s, s\ge 0\}$ is a fractional Brownian motion with $H\in(0,1/2)$. Then there exists a constant $\theta:=\theta(H,r)$ such that for all $a>0$ and all $r>0$, it
holds that
\begin{align}\label{E:LowBd}
\E\left(\left|\int_0^a f(s) \ud \widehat{B}_s \right|^r\right)\ge
\theta \Norm{f}_{L^{1/H}(0,a)}^r\qquad\text{for all $f\in \calH_a$.}
\end{align}
Moreover, if $f(s)$ is a process with value in a separable Hilbert space $V$,
one can view $f$ as a two-parameter process: $f:[0,a]\times D\ni (s,\omega)\mapsto f(s,\omega) \in \R$.
If $f(\cdot,\omega)\in \calH_a$ for all $\omega\in D$, then,
\begin{align}\label{E:LowBd2}
\E\left(\Norm{\int_0^a f(s) \ud \widehat{B}_s}_V^r\right)\ge
\theta \left(\int_0^a\Norm{f(s)}_V^{1/H}\ud s\right)^{rH}.
\end{align}
\end{lemma}
\begin{proof}
Because $\int_0^a f(s) \ud \widehat{B}_s $ is a centered Gaussian random variable,
there exists a finite constant $C_r>0$ such that
\[
\E\left[\left|\int_0^a f(s) \ud \widehat{B}_s \right|^r\right]
\ge C_r \left(\E\left[\left|\int_0^a f(s) \ud \widehat{B}_s \right|^2\right]\right)^{r/2}.
\]
Hence, we only need to prove the case where $r=2$.

We first note that \eqref{E:LowBd} is proved in part (i) of Theorem 1.2 in \cite{MMV00} for all $f$
that has bounded variation on $[0,a]$, and in particular, it holds for all simple functions.
Now fix $f\in\calH_a$. There exist simple functions $f_n$ on $[0,a]$ such that
$\Norm{f-f_n}_{\calH_a}\rightarrow 0$ as $n\rightarrow 0$.
Then
\begin{align}\label{E_:Fatou}
 \E\left[\left(\int_0^af(s)\ud \widehat{B}_s\right)^2\right]
 & = \lim_{n\rightarrow\infty}
 \E\left[\left(\int_0^af_n(s)\ud \widehat{B}_s\right)^2\right]\ge
 \lim_{n\rightarrow\infty}\theta \left(\int_0^a |f_n(s)|^{1/H} \ud s\right)^{2H}.
\end{align}
Because \eqref{E:LowBd} holds for simple functions, we see that
\[
\Norm{f_n-f_m}_{L^{1/H}(0,a)}\le
\Norm{f_n-f_m}_{\calH_a}.
\]
Thus, $\{f_n\}_{n\ge 1}$ is a Cauchy sequence in $L^{1/H}(0,a)$.
Hence, by passing to a subsequence when necessary, it implies that
$f_n\rightarrow f$ almost everywhere.
Therefore, \eqref{E:LowBd} is proved by applying Fatou's lemma
to the right-hand side  of \eqref{E_:Fatou}.

% \bigskip
Now if $f(s)$ is a process with value in a separable Hilbert space $V$,
let $\{e_i\}_{i\in\bbN}$ be a set of orthonormal basis of $V$.
Since $f(\cdot,\omega)\in \calH_a$ for all $\omega\in D$, we see that
$\InPrd{f(s,\cdot),e_i}_V\in\calH_a$. Hence, by \eqref{E:LowBd},
\begin{align*}
 \E\left(\Norm{\int_0^a f(s)\ud \widehat{B}_s}_V^2\right) & =
 \E\left(\Norm{\sum_{i=1}^\infty \int_0^a \InPrd{f(s),e_i}_V\ud \widehat{B}_s\: e_i}_V^2\right)\\
 &=\sum_{i=1}^\infty
 \E\left(\left|\int_0^a \InPrd{f(s),e_i}_V\ud \widehat{B}_s\right|^2\right)\\
 &\ge \theta\sum_{i=1}^\infty
 \left(\int_0^a \left|\InPrd{f(s),e_i}_V^2 \right|^{\frac{1}{2H}} \ud s\right)^{2H}
\\
&=
\theta\sum_{i=1}^\infty
 \Norm{ \InPrd{f(s),e_i}_V^2}_{L^{\frac{1}{2H}}(0,a)}\\
&\ge
\theta
 \Norm{ \sum_{i=1}^\infty \InPrd{f(s),e_i}_V^2}_{L^{\frac{1}{2H}}(0,a)}\\
&=\theta
 \Norm{ \Norm{f(s)}_V^2 }_{L^{\frac{1}{2H}}(0,a)},
\end{align*}
where we can apply Minkovski's inequality in the last inequality because $H\in (0,1/2)$.
This completes the proof of Lemma \ref{L:LowBd}.
\end{proof}

\bigskip
\begin{proof}[Proof of Theorem \ref{T:LowBd}]
Without loss of generality, we may assume that $u_0\equiv 1$.
From \eqref{E:u^k} and by Lemma \ref{L:WWYY}, we see that
\begin{align*}
\E\left[u(t,x)^k\right]
=\E^B\exp\left\{
 \E^{Y,\widehat{B}}\left[\left(\int_0^t \sum_{i=1}^k Y(B_{t-s}^{i,x})\ud \widehat{B}_s\right)^2\right]
\right\}.
\end{align*}
Then by \eqref{E:Holder-Y}, we see that
$s\mapsto \sum_{i=1}^k Y(B_{t-s}^{i,x})$ is $\gamma$-H\"older continuous a.s. for all $\gamma<\alpha/2$.
Since $\alpha/2>1/2-H$, one can find $\gamma'$ such that $1/2-H<\gamma'<\alpha/2$. Hence, by \eqref{E:Spaces},
$s\mapsto \sum_{i=1}^k Y(B_{t-s}^{i,x})$ is in $\calH_t$ for all realizations of $Y$.
Therefore, by Lemma \ref{L:LowBd}, for some constant $C_H'>0$,
\begin{align*}
 \E^{Y,\widehat{B}}\left[\left(\int_0^t \sum_{i=1}^k Y\left(B_{t-s}^{i,x}\right)\ud \widehat{B}_s\right)^2\right]
 &\ge C_H' \left(\int_0^t\E^{Y}\left[\left|\sum_{i=1}^k Y\left(B_{s}^{i,x}\right)\right|^2\right]^{\frac{1}{2H}} \ud s\right)^{2H}
 = C_H' I_t,
\end{align*}
where
\[
I_t:= \left(\int_0^t\left[\sum_{i,j=1}^k Q\left(B_{s}^{i,x},B_{s}^{j,x}\right)\right]^{\frac{1}{2H}} \ud s\right)^{2H}.
\]
Then for any $M>0$ (to be chosen later), by condition \eqref{E:H2} and by writing $B_s^{i,x}=(B_s^{i,x_1,1},\dots,B_s^{i,x_d,d})$,
\begin{align*}
\E\left[u(t,x)^k\right]
&\ge \E^B\exp\left(C_H' I_t\right)\\
&\ge
\bbP\left(B_s^{i,x_j,j}>M,\:\forall s\in[t/2,t], \forall i=1,\dots,k,\: \forall j=1,\dots,d\: \right) \exp\left(C_H k^2 M^{2\beta} t^{2H}\right)\\
&\ge
\bbP\left(B_s^{1,y,1}>M,\:\forall s\in[t/2,t]\right)^{kd}
\exp\left(C_H k^2 M^{2\beta} t^{2H}\right),
\end{align*}
where $C_H=C_H'C_2$ and
\[
y=\min_{i=1,\dots, d} x_i. 
\]
In the following, for simplicity, we use $B_t$ to denote the one-dimensional standard Brownian motion
starting from the origin.
Hence,
\[
\E\left[u(t,x)^k\right]
\ge
\bbP\left(B_s+y>M,\:\forall s\in[t/2,t]\right)^{kd}
\exp\left(C_H k^2 M^{2\beta} t^{2H}\right).
\]
Assume that $M\ge |y|$. Then
\begin{align*}
\bbP\left(B_s+y>M,\:\forall s\in[t/2,t]\right)
&\ge \bbP\left(B_s>2M,\:\forall s\in[t/2,t]\right)\\
&\ge \bbP\left(B_s>\frac{2M}{\sqrt{t}},\:\forall s\in[1/2,1]\right)\\
&\ge \bbP\left(B_{1/2}>\frac{4M}{\sqrt{t}},\:|B_{s}-B_{1/2}|<\frac{2M}{\sqrt{t}}, \: \forall s\in[1/2,1]\right)\\
&=\bbP\left(B_{1/2}>\frac{4M}{\sqrt{t}}\right)
\bbP\left(\sup_{s\in[0,1/2]}|B_{s}|<\frac{2M}{\sqrt{t}}\right),
\end{align*}
where we have used the scaling property of the Brownian motion.
By a standard argument
\[
\bbP\left(B_1>r\right)^2\ge \frac{1}{2\pi}\int_0^{\pi/2}\ud\theta \int_{\sqrt{2} r}^\infty e^{-\frac{s^2}{2}}s\ud s=\frac{1}{4}e^{-r^2}, \quad(r>0)
\]
we have that
\begin{align*}
\bbP\left(B_{1/2}>\frac{4M}{\sqrt{t}}\right)&=
\bbP\left(B_1>\frac{4\sqrt{2}M}{\sqrt{t}}\right)\ge 2^{-1} \exp\left(-\frac{16 M^2}{t}\right).
\end{align*}
% where in the last step we have applied Lemma \ref{L:NormLow},
% and $\Theta$ is some universal constant.
By Chebyshev's inequality and Fernique's theorem, for some $\lambda>0$,
\begin{align*}
 \bbP\left(\sup_{s\in[0,1/2]}|B_{s}|<\frac{2M}{\sqrt{t}}\right)
 &=1-\bbP\left(\sup_{s\in[0,1/2]}|B_{s}|>\frac{2M}{\sqrt{t}}\right)\ge
 1- C_\lambda e^{-4\lambda M^2/t},
\end{align*}
where $C_\lambda=\E\exp\left(\lambda \sup_{s\in[0,1/2]}|B_s|^2\right)<\infty$.
Now assume that $M/\sqrt{t}$ is sufficiently large such that
\begin{align}\label{E:MT}
\left(1- C_\lambda  e^{-4\lambda M^2/t}\right)^{kd}
\ge 1/2.
\end{align}
Therefore, provided that $M\ge |y|$ and \eqref{E:MT} is true, we have that
\begin{align*}
\E\left[u(t,x)^k\right]&\ge 2^{-(kd+1)}\exp\left(C_H k^2 M^{2\beta} t^{2H}-\frac{16 kM^2}{t}\right).
\end{align*}
Now we maximize
\[
f(M)=C_H k^2 M^{2\beta} t^{2H}-\frac{16 kM^2}{t},\quad\text{for $M\ge 0$}.
\]
By solving $f'(M)=0$, we see that $f$ is maximized at
\[
M_0= \left(16^{-1} \beta\: k\: C_H t^{1+2H}\right)^{\frac{1}{2(1-\beta) }}
\]
with 
\[
\sup_{M\ge 0} f(M) = f(M_0) =16^{\frac{\beta }{\beta -1}}
   (1-\beta) \beta^{\frac{\beta}{1-\beta}}   C_H^{\frac{1}{1-\beta
   }} k^{\frac{2-\beta}{1-\beta}} t^{\frac{\beta+2H}{1-\beta}}.
\]
Clearly, when either $k$ or $t$ is sufficiently large, the condition $M_0\ge |y|$ is satisfied.
Similarly, because $M_0/\sqrt{t}=\left(16^{-1} \beta k C_H\right)^{\frac{1}{2-2 \beta}}t^{\frac{\beta+2H}{2(1-\beta)}}$,
when $t$ is large enough, \eqref{E:MT} is also satisfied.
This completes the proof of Theorem \ref{T:LowBd}.
\end{proof}

\section*{Appendix}

\begin{lemma}\label{L:Gamma}
For all $a, b, u,v,w >0$, if $u+v\le w+1/2$ and $w>1/2$, then
\[
\sup_{n\in\bbN}\frac{\Gamma(an+u)\Gamma(bn+v)}{\Gamma((a+b)n+w)}<\infty.
\]
\end{lemma}
\begin{proof}
 We only need  to show that
\[
\lim_{n\rightarrow \infty}\frac{\Gamma(an+u)\Gamma(bn+v)}{\Gamma((a+b)n+w)}<\infty.
\]
By Stirling's formula (see \cite[5.11.3 or 5.11.7]{NIST2010}), as $n$ is large, we see that
\begin{align*}
\frac{\Gamma(an+u)\Gamma(bn+v)}{\Gamma((a+b)n+w)}\approx
\sqrt{\pi}\exp\Big\{&\quad
(an+u-1/2)\log(an)\\
&+(bn+v-1/2)\log(bn)\\
&-((a+b)n+w-1/2)\log((a+b)n)
\Big\}.
\end{align*}
Denote the right-hand side of the above quantity by $I_n$.
By the supper-additivity of $f(x)=x\log x$, namely $f(x+y)\ge f(x)+f(y)$ for all $x,y\ge 0$, we see that
\[
I_n\le
\sqrt{\pi}\exp\Big\{
(u-1/2)\log(an)+
(v-1/2)\log(bn)-(w-1/2)\log((a+b)n)
\Big\}.
\]
Because $w>1/2$, we can apply the inequality $\log((a+b)n)\ge \frac{1}{2}[\log(an)+\log(bn)]$ to obtain that
\begin{align*}
I_n\le&
\sqrt{\pi}\exp\Big\{
(u-1/4-w/2)\log a+
(v-1/4-w/2)\log b
+(u+v-w-1/2)\log n
\Big\}\\
=& C n^{u+v-w-1/2} \le C,\quad\text{for all $n\in\bbN$,}
\end{align*}
where the last inequality is due to the assumption that $u+v-w-1/2\le 0$.
\end{proof}

Let $E_{\alpha,\beta}(z)$ be the {\it Mittag-Leffler function}
 \[
E_{\alpha,\beta}(z)=\sum_{n=0}^\infty\frac{z^n}{\Gamma(\alpha n+\beta)}, \quad \Re\alpha>0,\:  \beta\in\bbC,\: z\in\bbC.
\]
\begin{lemma}[Theorem 1.3 p.~32 in \cite{Podlubny99FDE}]\label{L:Eab}
If $0<\alpha<2$, $\beta$ is an arbitrary complex number and $\mu$ is an
arbitrary real number such that
\[
\pi\alpha/2<\mu<\pi \wedge (\pi\alpha)\;,
\]
then for an arbitrary integer $p\ge 1$ the following expression holds:
\[
E_{\alpha,\beta}(z) = \frac{1}{\alpha}
z^{(1-\beta)/\alpha} \exp\left(z^{1/\alpha}\right)
-\sum_{k=1}^p \frac{z^{-k}}{\Gamma(\beta-\alpha k)} + O\left(|z|^{-1-p}\right)
,\quad |z|\rightarrow\infty,\quad |\arg(z)|\le \mu\:.
\]
\end{lemma}
\begin{lemma}\label{L:ML-bds}
For all $\alpha>0$ and $\beta\le 1$,  there exists some constant $C=C_{\alpha,\beta}\ge 1$ such that
\[
E_{\alpha,\beta}(z)
\le C \exp\left\{C z^{1/\alpha}\right\},\quad \text{for all $z\ge 0$}.
\]
\end{lemma}
\begin{proof}
By Lemma \ref{L:Eab}, we see that for some constants $C_{\alpha,\beta}'>0$ and $C_{\alpha,\beta}\ge 1$,
\[
E_{\alpha,\beta}(z)\le C_{\alpha,\beta}'(1+z^{(1-\beta)/\alpha} \exp\left(z^{1/\alpha}\right))
\le
C_{\alpha,\beta}\exp\left(C_{\alpha,\beta} z^{1/\alpha}\right),
\]
for all $z\ge 0$.
\end{proof}

\addcontentsline{toc}{section}{Bibliography}
\def\polhk#1{\setbox0=\hbox{#1}{\ooalign{\hidewidth
  \lower1.5ex\hbox{`}\hidewidth\crcr\unhbox0}}} \def\cprime{$'$}


\begin{thebibliography}{10}

\bibitem{BC15AOP}
Balan, Raluca M. and Conus, Daniel.
\newblock Intermittency for the wave and heat equations with fractional noise in time.
\newblock {\it Ann. Probab. } to appear (2015).

\bibitem{BertiniCancrini95}
Bertini, Lorenzo and Cancrini, Nicoletta.
\newblock The stochastic heat equation: Feynman-Kac formula and intermittence.
\newblock {\it J. Statist.  hys.} 78 (1995), no. 5-6, 1377--1401.

% \bibitem{AiraultRenZhang00}
% H.~Airault, J.~Ren, and X.~Zhang.
% \newblock Smoothness of local times of semimartingales.
% \newblock {\em C. R. Acad. Sci. Paris S\'er. I Math.} 330 (2000), no. 8, 719--724.

% \bibitem{BCR11FNPDE}
% H.~Bahouri, J.Y.~Chemin and R.~Danchin.
% \newblock {\em Fourier analysis and nonlinear partial differential equations.}
% \newblock Grundlehren der Mathematischen Wissenschaften [Fundamental Principles of Mathematical Sciences], 343. Springer, Heidelberg, 2011. xvi+523 pp.

\bibitem{CarmonaMolchanov94PAM}
Carmona, Ren{\'e} A. and Molchanov, S. A.
\newblock Parabolic {A}nderson problem and intermittency.
\newblock {\em Mem. Amer. Math. Soc.}, 108(518), 1994.

\bibitem{ChenDalang13Heat}
Chen, Le and Dalang, Robert C.
\newblock Moments, intermittency, and growth indices for the nonlinear
stochastic heat equation with rough initial conditions.
\newblock {\em Ann. Probab.} 43 (2015), no. 6, 3006--3051.
%
% \bibitem{ChenDalang15FracHeat}
% Chen, Le and Dalang, Robert C.
% \newblock  Moments, intermittency and growth indices for nonlinear stochastic fractional heat equation,
% \newblock {\em  Stoch. Partial Differ. Equ. Anal. Comput.} 3 (2015), no. 3, 360--397.

\bibitem{ChenHuNualart15}
Chen, Le and Hu, Yaozhong and Nualart, David.
\newblock  Nonlinear stochastic time-fractional slow and fast diffusion equations on $\R^d$,
\newblock {\em Preprint arXiv:1509.07763} (2015).

%
% \bibitem{ChenKim14Comparison}
% L.~Chen and K.~Kim.
% \newblock On comparison principle and strict positivity of solutions to the nonlinear stochastic fractional heat equations.
% \newblock {\em submitted}, arXiv:1410.0604, 2014.

% \bibitem{ChungWilliams90}
% K.~L. Chung and R.~J. Williams.
% \newblock {\em Introduction to stochastic integration}~(2$^{\text{nd}}$ ed.).
% \newblock Birkh\"auser Boston Inc., Boston, MA, 1990.

\bibitem{CJKS13}
Conus, Daniel and Joseph, Mathew and Khoshnevisan, Davar and Shiu, Shang-Yuan.
\newblock On the chaotic character of the stochastic heat equation, II.
\newblock {\em Probab. Theory Related Fields} 156 (2013), no. 3-4, 483--533.


%
% \bibitem{Erdelyi1954-I}
% A.~Erd{\'e}lyi, W.~Magnus, F.~Oberhettinger, and F.~G. Tricomi.
% \newblock {\em Tables of integral transforms. {V}ol. {I}}.
% \newblock McGraw-Hill Book Company, Inc., New York-Toronto-London, 1954.

\bibitem{FK08Int}
Foondun, Mohammud and Khoshnevisan, Davar.
\newblock Intermittence and nonlinear parabolic stochastic partial differential equations.
\newblock {\em Electron. J. Probab.}  14 (2009), 548--568.


% \bibitem{MuellerNualart08}
% C.~Mueller and D.~Nualart.
% \newblock Regularity of the density for the stochastic heat equation.
% \newblock {\em Electron. J. Probab.} 13 (2008), no. 74, 2248--2258.

\bibitem{HHNT15EJP}
Hu, Yaozhong and Huang, Jingyu and Nualart, David and Tindel, Samy.
\newblock Stochastic heat equations with general multiplicative Gaussian noises: H\"{o}lder continuity and intermittency
\newblock {\it Electron. J. Probab.} 20 (2015), no. 55, 50 pp.


%
% \bibitem{HL14}
% Y.~Hu and K~. Le.
% \newblock Nonlinear Young integrals and differential systems in H\"older media.
% \newblock Preprint at {\it arXiv:1404.7582}, 2014.


\bibitem{HLN12AOP}
Hu, Yaozhong and Lu, Fei and Nualart, David.
\newblock Feynman-Kac formula for the heat equation driven by fractional noise with Hurst parameter $H<1/2$
\newblock {\it Ann. Probab.} 40 (2012), no. 3, 1041--1068.

\bibitem{HN09PTRF}
Hu, Yaozhong and Nualart, David.
\newblock Stochastic heat equation driven by fractional noise and local time.
\newblock {\em Probab. Theory Related Fields}, 143 (2009), no. 1--2, 28--328.

\bibitem{HNS09}
Hu, Yaozhong and Nualart, David and Song, Jian.
\newblock Fractional martingales and characterization of the fractional Brownian motion.
\newblock {\em Ann. Probab.} 37 (2009), no. 6, 2404--2430.

\bibitem{KM15}
Kalbasi, Kamran and Mountford, Thomas S.
\newblock Feynman-Kac representation for the parabolic Anderson model driven by fractional noise.
\newblock {\em J. Funct. Anal.} 269 (2015), no. 5, 1234--1263.

\bibitem{MMV00}
M{\'e}min, Jean and Mishura, Yulia and Valkeila, Esko.
\newblock
Inequalities for the moments of Wiener integrals with respect to a fractional Brownian motion.
\newblock {\em Statist. Probab. Lett.} 51 (2001), no. 2, 197--206.

% \bibitem{Nualart09}
% D.~Nualart.
% \newblock {\em Malliavin calculus and its applications.}
% \newblock CBMS Regional Conference Series in Mathematics, 110. Published for the Conference Board of the Mathematical Sciences, Washington, DC; by the American Mathematical Society, Providence, RI,
% 2009.

\bibitem{Nualart06}
Nualart, David.
\newblock {\em The Malliavin calculus and related topics} (Second edition).
\newblock Probability and its Applications (New York). Springer-Verlag, Berlin, 2006.

% \bibitem{NualartQuer07}
% D.~Nualart and L.~Quer-Sardanyons.
% \newblock Existence and smoothness of the density for spatially homogeneous {SPDEs}.
% \newblock {\em Potential Anal.} 27 (2007), no. 3, 281--299.

% \bibitem{NualartVives92}
% D.~Nualart, and J.~Vives.
% \newblock Smoothness of Brownian local times and related functionals.
% \newblock {\em Potential Anal.} 1 (1992), no. 3, 257--263.

\bibitem{NIST2010}
Olver, Frank W. J. and Lozier, Daniel W. and Boisvert, Ronald F. and Clark, Charles W.
\newblock {\em N{IST} handbook of mathematical functions}.
\newblock U.S. Department of Commerce National Institute of Standards and Technology, Washington, DC, 2010.

\bibitem{P-Taqqu01}
Pipiras, Vladas and Taqqu, Murad S.
\newblock Are classes of deterministic integrands for fractional Brownian motion on an interval complete?
\newblock {\em Bernoulli} 7 (2001), no. 6, 873--897.

% \bibitem{P-Taqqu00PTRF}
% V.~Pipiras and M.~Taqqu.
% \newblock Integration questions related to fractional Brownian motion.
% \newblock {\em Probab. Theory Related Fields} 118 (2000), no. 2, 251--291.

\bibitem{Podlubny99FDE}
Podlubny, Igor.
\newblock {\em Fractional differential equations}.
\newblock Academic Press Inc., San Diego, CA, 1999.

% \bibitem{RevuzYor99}
% D.~Revuz, and M.~Yor.
% \newblock {\em Continuous martingales and Brownian motion.} Third edition.
% \newblock Springer-Verlag, Berlin, 1999.

\bibitem{Zel90}
Zel{\cprime}dovich, Ya. B. and Ruzma{\u\i}kin, A. A. and Sokoloff, D. D.
\newblock {\em The almighty chance.}
\newblock Translated from the Russian by Anvar Shukurov.
\newblock World Scientific Lecture Notes in Physics, 20. World Scientific Publishing Co., Inc., River Edge, NJ, 1990.

\end{thebibliography}
\end{document}